\documentclass{amsart}

\usepackage{amsmath}
\usepackage{amsfonts}
\usepackage{amssymb,enumerate}
\usepackage{amsthm}
\usepackage[all]{xy}
\usepackage{hyperref}

\newcommand{\Ann}{\operatorname{Ann}}
\newcommand{\Hom}{\operatorname{Hom}}
\newcommand{\Ext}{\operatorname{Ext}}

\newcommand{\Supp}{\operatorname{Supp}}

\newcommand{\grade}{\operatorname{grade}}
\newcommand{\pgrade}{\operatorname{p.grade}}
\newcommand{\h}{\operatorname{ht}}

\newcommand{\fm}{\frak{m}}
\newcommand{\fp}{\frak{p}}

\newcommand{\fq}{\frak{q}}
\newtheorem{thm}{Theorem}[section]
\newtheorem{cor}[thm]{Corollary}
\newtheorem{lem}[thm]{Lemma}
\newtheorem{prop}[thm]{Proposition}
\newtheorem{defn}[thm]{Definition}

\newtheorem{rem}[thm]{Remark}

\begin{document}

\bibliographystyle{amsplain}

\date{}

\author{A. Mahdikhani, P. Sahandi and N. Shirmohammadi}

\address{Department of Mathematics, University of Tabriz, Tabriz,
Iran} \email{a.mahdikhani@tabrizu.ac.ir}

\address{Department of Mathematics, University of Tabriz, Tabriz,
Iran} \email{sahandi@ipm.ir}

\address{Department of Mathematics, University of Tabriz, Tabriz, Iran.}
\email{shirmohammadi@tabrizu.ac.ir}



\title[Cohen-Macaulayness of trivial extensions]{Cohen-Macaulayness of trivial extensions}

\begin{abstract}
Our goal is to determine when the trivial extensions of
commutative rings by modules are Cohen-Macaulay in the
sense of Hamilton and Marley. For this purpose, we
provide a generalization of the concept of
Cohen-Macaulayness of rings to modules.
\end{abstract}

\maketitle

\section{Introduction}

Throughout this paper all rings are commutative, with identity, and all modules are
unital. The theory of Cohen-Macaulay rings admits a rich
theory in commutative \emph{Noetherian} rings. No attempts
have been made to develop the concept of Cohen-Macaulayness
to non-Noetherian rings until 1992. Then,
Glaz \cite{G0} begun an investigation on the notion of
Cohen-Macaulayness for non-Noetherian rings and, she
asked how one can define a non-Noetherian notion
of Cohen-Macaulayness such that the definition
coincides with the original one in the Noetherian
case, and that coherent regular rings are Cohen-Macaulay, see \cite[p. 220]{G}.

More recently, Hamilton and Marley \cite{HM} established a
definition of Cohen-Macaulayness for non-Noetherian rings. More
precisely, employing Schenzel's weakly proregular sequences
\cite{Sch}, they used the tool of \v{C}ech cohomology modules to
define the notion of parameter sequences. A parameter sequence such
that every truncation on the right is also a parameter sequence is
called a strong parameter sequence. In some sense, this is a
generalization of the system of parameters to the non-Noetherian
case. They then called a ring \emph{Cohen-Macaulay} if every strong
parameter sequence is a regular sequence. They showed that their
definition coincides with the original one in the Noetherian case
and that coherent regular rings are Cohen-Macaulay (in the sense of
new definition).

Let $R$ be a ring and $M$ be an $R$-module. In 1955, Nagata
construct a ring extension of $R$ called the \emph{trivial
extension} of $R$ by $M$, denoted here by $R\ltimes M$. This ring is
of particular importance in commutative algebra (cf. \cite{AW} and
\cite[Theorem 3.3.6]{BH}). One of the properties of trivial
extension is as follows: if $R$ is Noetherian local and $M$ is
finitely generated, then $R\ltimes M$ is Cohen-Macaulay if and only
if $R$ is Cohen-Macaulay and $M$ is maximal Cohen-Macaulay, see
\cite[Corollary 4.14]{AW}. Motivated by this and \cite[Example
4.3]{HM}, we wish to investigate whether the trivial extension
$R\ltimes M$ is Cohen-Macaulay with the definition of Hamilton and
Marley. The main ingredient in this investigation is to formulate a
definition for a module to be Cohen-Macaulay, which was left
unaddressed in \cite{HM}. So we find ourselves forced to extend the
notion of Cohen-Macaulay to modules.

The outline of the paper is as follows. In Section 2, we recall some
essential definitions and results on which we base our approach. In
Section 3, after defining weakly proregular sequences on modules, we
give a characterization of such sequences using the vanishing of
suitable \v{C}ech cohomology modules. We then define (strong)
parameter sequences on modules. Continually, after citing some
elementary properties of such sequences, we relate these sequences
to system of parameters of finitely generated modules in Noetherian
local rings. In Section 4, we define Cohen-Macaulayness of modules
over non-Noetherian rings. Among other things, we are able to
establish our main result which says when the trivial extension
$R\ltimes M$ is Cohen-Macaulay.

\section{Preliminaries}

Let $R$ be a ring, $I$ be an ideal of $R$ and  $M$ be an $R$-module.
Following \cite{BH}, a sequence ${\bf x}:=x_{1},\ldots,x_{\ell} \in
R$ is called a \emph{weak $M$-regular sequence} if $x_{i}$ is a
non-zero-divisor on $M/(x_{1},\ldots,x_{i-1})M$ for
$i=1,\ldots,\ell$. If, in addition, $M\neq{\bf x}M$, we call ${\bf
x}$ an \emph{$M$-regular sequence}. The \emph{classical grade} of
$I$ on $M$, denoted $\grade(I,M)$, is defined to be the supremum of
the lengths of all weak $M$-regular sequences contained in $I$, see
\cite{Ho}.

The \emph{polynomial grade} of $I$ on $M$ is defined by
\begin{displaymath}
\pgrade_R(I,M):=\lim_{m \to \infty }\grade(IR[t_{1},\ldots
,t_{m}],R[t_{1},\ldots ,t_{m}]\otimes_{R}M).
\end{displaymath}
It follows from \cite{A} and \cite{Ho} that
$$ \pgrade_R(I,M)=\sup \{ \grade(IS,S \otimes_{R} M)|S\text{ is a faithfully flat }R\text{-algebra} \} .$$

Let $x$ be an element of $R$. Let $C(x)$ denote the complex $0\to
R\to R_x\to 0$ where the differential is the natural localization
map. For a sequence ${\bf x}:=x_{1},\ldots,x_{\ell}$ of elements of
$R$, the \v{C}ech complex $C({\bf x})$ is inductively defined by
$C({\bf x}):=C(x_{1},\ldots,x_{\ell-1})\otimes_{R}C(x_{\ell})$. Then
we set $C({\bf x};M):=C({\bf x})\otimes_{R}M$. The \emph{$i$th
\v{C}ech cohomology} $H_{{\bf x}}^{i}(M)$ of $M$ with respect to the
sequence ${\bf x}$ is defined to be the $i$th cohomology of $C({\bf
x};M)$.

The following summarizes some essential properties of \v{C}ech cohomology modules.

In the sequel, for a finite sequence ${\bf x}$ of elements
of the ring $R$, $\ell({\bf x})$ denotes the length of ${\bf x}$.

\begin{prop}\label{ceck}
(see \cite[Proposition 2.1]{HM}) Let $R$ be a ring, ${\bf x}$
a finite sequence of elements of $R$ and $M$ an $R$-module.
\begin{enumerate}
 \item If ${\bf y}$ is a finite sequence of elements of $R$ such that
$\sqrt{{\bf y}R}=\sqrt{{\bf x}R}$, then $H^i_{\bf y}(M)\cong H^i_{\bf x}(M)$ for all $i$.
 \item (Change of rings) Let $f :R \to S$ be a ring homomorphism and $N$ an $S$-module. Then
$H^i_{\bf x}(N)\cong H^i_{f({\bf x})}(N)$ for all $i$.
\item (Flat base change) Let $f :R \to S$ be a flat ring homomorphism. Then
$H^i_{\bf x}(M)\otimes_RS\cong H^i_{\bf x}(M\otimes_RS)\cong H^i_{f({\bf x})}(M\otimes_RS)$ for all $i$.
\item $H^{\ell({\bf x})}_{\bf x}(M)\cong H^{\ell({\bf x})}_{\bf x}(R)\otimes_RM$.
\end{enumerate}
\end{prop}

Applying parts (2) and (4) of Proposition \ref{ceck} to the natural
homomorphism $R\to R/\Ann M$ together with \cite[Proposition
2.7]{HM} yields the next corollary.

\begin{cor}\label{dim}
Let $R$ be a ring and $M$ be a finitely generated
$R$-module of finite dimension. Then $H^{i}_{{\bf x}}(M)=0$ for all
$i>\dim M$.
\end{cor}

The $i$th \emph{local cohomology} $H^i_I(M)$ of $M$ with respect to $I$ is defined by
$$H^i_I(M):=\varinjlim\Ext^i_R(R/I^n,M).$$
In the case that $R$ is Noetherian, one has $H^i_I(M)=H^i_{\bf
x}(M)$ for all $i$, where $I:={\bf x}R$, see \cite[Theorem
5.1.20]{BSh}.

Let $\fp$ a prime ideal of $R$. By $\h_M\fp$, we mean the Krull
dimension of the $R_\fp$-module $M_\fp$. Also, for an ideal $I$ of $R$
$$\h_MI:=\inf \{\h_M\fp|\fp \in \Supp M \cap V(I) \}.$$

\section{Weakly proregular and parameter sequences}

\subsection{Weakly proregular sequences}

It is mentioned in \cite{HM} that local cohomology and \v{C}ech
cohomology are not in general isomorphic over non-Noetherian rings.
In \cite{Sch}, Schenzel gave necessary and sufficient conditions on
a sequence ${\bf x}$ of a ring $R$ such that the isomorphism
$H^i_I(M)\cong H^i_{\bf x}(M)$ holds for all $i$ and $R$-modules
$M$, where $I:={\bf x}R$. Such sequences are called $R$-weakly
proregular sequences. In the following, we provide the module
theoretic version of this notion.

Let $R$ be a ring and $M$ be an $R$-module. For $x \in R$,
we use $\mathbb{K}_{\bullet}(x,M)$ to denote the Koszul complex
$$
0\longrightarrow M \stackrel{x}\longrightarrow M \longrightarrow 0.
$$
For a sequence ${\bf x} =x_{1},\ldots ,x_{\ell}$ the Koszul complex
$\mathbb{K}_{\bullet}({\bf x},M)$ is defined to be the complex
$\mathbb{K}_{\bullet}(x_1,M)\otimes_{R} \cdots \otimes_{R}
\mathbb{K}_{\bullet}(x_{\ell},M)$. The $i$th homology of
$\mathbb{K}_{\bullet}({\bf x},M)$ is denoted by $H_{i}({\bf x},M)$
and called the Koszul homology of the sequence ${\bf x}$ with
coefficients in $M$. For $m \geq n$, there exists a chain map
$\phi_{n}^{m}({\bf x},M): \mathbb{K}_{\bullet}({\bf x}^{m},M)\to
\mathbb{K}_{\bullet}({\bf x}^{n},M)$ given by
$$
\phi_{n}^{m}({\bf x},M)=\phi_{n}^{m}(x_{1},M)\otimes_{R} \cdots
\otimes_{R} \phi_{n}^{m}(x_{l},M),
$$
where, for all $x\in R$, $\phi_{n}^{m}(x,M)$
 is the chain map\\
 \begin{displaymath}
 \xymatrix{0\ar[r] &
            M \ar[r]^{x^{m}}\ar[d]_{x^{m-n}} &
            M \ar[r]\ar[d]^{=} & 0 \\
 0\ar[r]& M \ar[r]^{x^n}&
               M\ar[r] & 0.}
\end{displaymath}
Hence, $\{\mathbb{K}_{\bullet}({\bf x}^{m},M),\phi_{n}^{m}({\bf x},M)\}$ is an inverse system of
complexes. Note, for each $i$, the map $\phi_{n}^{m}({\bf x},M)_i$
induces a homomorphism of homology modules $H_{i}({\bf x}^m,M)\to H_{i}({\bf x}^n,M)$.
We also denote this induced homomorphism by $\phi_{n}^{m}({\bf x},M)_i$.
The sequence $ {\bf x} =x_{1},\ldots ,x_{\ell} $ is
called \emph{$M$-weakly proregular} if, for each $n$, there exists an
$ m \geq n $ such that the map $\phi_{n}^{m}({\bf x},M)_i:
H_{i}({\bf x}^{m},M)\to H_{i}({\bf x}^{n},M)$ is zero for all $i
\geq 1$ (see \cite{Sch}).
Note that an element $ x \in R $ is $M$-weakly proregular if and only
 if there exists an $ n\geq 1 $ such that $(0:_{M}x^{n})=(0:_{M}x^{n+1})$.

\begin{rem}\label{3.2}
Let ${\bf x}$ be a
finite sequence of elements of $R$.
\begin{enumerate}
 \item If ${\bf x}$ is an $M$-weakly proregular
 sequence, then so  is any permutation of ${\bf x}$.
 \item Any $M$-regular sequence is $M$-weakly proregular.
\end{enumerate}
\end{rem}

The following result provides another description of weakly
proregular sequences using \v{C}ech cohomology. Its proof is
inspired by the proof of \cite[Lemma 2.4]{Sch}. Here, for a sequence
${\bf x} =x_{1},\ldots ,x_{\ell}$, we use $H^{i}({\bf x},M)$ to
denote the $i$th cohomology of the complex
$\Hom_R(\mathbb{K}_{\bullet}({\bf x}),M)$, where
$\mathbb{K}_{\bullet}({\bf x}):=\mathbb{K}_{\bullet}({\bf x},R)$,
and we call it the $i$th Koszul cohomology of the sequence ${\bf x}$
with coefficients in $M$. It follows from \cite[Theorem 5.2.5]{BSh}
that $H^{i}_{{\bf x}}(M)\cong \varinjlim H^{i}({\bf x}^n,M)$.

\begin{thm}\label{3.3}
Let ${\bf x}$ be a finite
sequence of elements of $R$. Then the following conditions
are equivalent:
\begin{enumerate}
  \item ${\bf x}$ is $M$-weakly proregular.
  \item $H_{{\bf x}}^{i}(\Hom_{R}(M,E))=0 $ for all injective $R$-modules $E$ and
$i\neq0$.
\end{enumerate}
\end{thm}
\begin{proof}
Assume that ${\bf x}$ is $M$-weakly proregular and that $E$ is an
injective $R$-module. Then $\Hom_R(\mathbb{K}_{\bullet}({\bf
x}^{n}),\Hom_{R}(M,E))\cong \Hom_R(\mathbb{K}_{\bullet}({\bf
x}^{n})\otimes_R M,E)$. So
\begin{align*}
H^{i}({\bf x}^{n},\Hom_{R}(M,E))
& = H^{i}(\Hom_R(\mathbb{K}_{\bullet}({\bf x}^{n}),\Hom_{R}(M,E)))\\
 & \cong  H^{i}(\Hom_R(\mathbb{K}_{\bullet}({\bf x}^{n})\otimes_R M, E))\\
 &\cong   \Hom_R(H_{i}({\bf x}^{n},M),E )
\end{align*}
for all $i$. Hence
\begin{displaymath}
\varinjlim { H^{i}({\bf x}^{n},\Hom_{R}(M,E))}
\cong\varinjlim{\Hom_R(H_{i}({\bf x}^{n},M),E)}.
\end{displaymath}
By assumption, for all $n\in \mathbb{N}$, the homomorphism
$$
{\Hom_R(H_{i}({\bf x}^{n},M),E)}\longrightarrow {\Hom_R(H_{i}({\bf x}^{m},M),E)}
$$
is zero, for some $m\geq n$. Therefore
\begin{displaymath}
\varinjlim { H^{i}({\bf x}^{n},\Hom_{R}(M,E))}
=\varinjlim{\Hom_R(H_{i}({\bf x}^{n},M),E)}=0.
\end{displaymath}
So that $H_{{\bf x}}^{i}(\Hom_{R}(M,E))=0 $ for $i\neq0$.

Conversely, assume that $H_{{\bf x}}^{i}(\Hom_{R}(M,E))=0 $ for all
injective $R$-modules $E$ and $i\neq0$. Let $f:H_{i}({\bf
x}^{n},M)\to E$ denote an injection of $H_{i}({\bf x}^{n},M)$ into
an injective $R$-module $E$. Then $f\in {\Hom_R(H_{i}({\bf
x}^{n},M),E)\cong H^{i}({\bf x}^{n},\Hom_{R}(M,E))}$. Since
\begin{displaymath}
\varinjlim { H^{i}({\bf x}^{n},\Hom_{R}(M,E))}=0,
\end{displaymath}
then there exists an $m\geq n$ such that $\Hom_R({\phi}_{n}^{m}({\bf x},M),1_{E})=0$ which
means $1_{E}f{\phi}_{n}^{m}({\bf x},M)=0$. So ${\phi}_{n}^{m}({\bf x},M)=0$,
because $f$ is injective.
\end{proof}

The above theorem together with Proposition \ref{ceck}
immediately yields the following corollary.

\begin{cor} \label{3.4}
Assume that ${\bf x}$ and ${\bf y}$ are
finite sequences of $R$ such that $\sqrt{{\bf x}R}=\sqrt{{\bf y}R}$.
Then ${\bf x}$ is $M$-weakly proregular if and only if ${\bf y}$ is $M$-weakly proregular.
\end{cor}

As a consequence of the following theorem one obtains that any finite sequence of elements in a
Noetherian ring is weakly proregular on any finitely generated module.

\begin{thm}\label{3.6}
Let $R$ be a Noetherian ring, $I$ be an ideal of $R$ and $M$ be a
finitely generated $R$-module. Then $H_{I}^{i}(\Hom_{R}(M,E))=0$ for
all injective $R$-modules $E$ and $i\neq 0$.
\end{thm}
\begin{proof}
Assume that $E$ is an injective $R$-module. Since the exact sequence
$0\to H^0_I(E)\to E\to E/H^0_I(E)\to 0$ is split by \cite[Corollary
2.1.5]{BSh}, one has $E\cong H^0_I(E)\oplus E/H^0_I(E)$. So that,
for all $i$, we have
$$
H_I^{i}(\Hom_{R}(M,E))\cong H_I^{i}(\Hom_{R}(M,H^0_I(E)))\oplus
H_I^{i}(\Hom_{R}(M,E/H^0_I(E))).
$$
It is easy to see that $\Hom_{R}(M,H^0_I(E))$ is $I$-torsion. Hence
$H_I^{i}(\Hom_{R}(M,H^0_I(E)))=0$ for all $i\neq 0$. Since
$E/H^0_I(E)$ is an injective $R$-module and $H^0_I(E/H^0_I(E))=0$,
then to complete the proof it is enough for us to show that
$H_I^{i}(\Hom_{R}(M,E))=0$ for the injective $R$-module $E$ with
additional condition that $E$ is $I$-torsion-free. For this, let
$$
\cdots \longrightarrow F_{1}\longrightarrow F_{0}
\longrightarrow M \longrightarrow 0
$$
be a free resolution of $M$. Then
$$
0\longrightarrow \Hom_{R}(M,E) \longrightarrow \Hom_{R}(F_{0},E)
\longrightarrow \Hom_{R}(F_{1},E) \longrightarrow \cdots
$$
is an augmented injective resolution of $\Hom_{R}(M,E)$ such that
$H^0_I(\Hom_{R}(F_i,E))=0$ for all $i$ since $H^0_I(E)=0$.
Therefore $H_I^{i}(\Hom_{R}(M,E))=0$ for all $i\neq 0$.
\end{proof}

The following lemma will be used later.

\begin{lem} \label{fwp}
Suppose that $f:R\to S$ is a flat ring homomorphism
and that $M$ is an $R$-module. If
${\bf x}$ is $M$-weakly proregular, then $f({\bf x})$
is $M\otimes_{R}S$-weakly proregular. The converse holds if $f$ is
faithfully flat.
\end{lem}
\begin{proof}
This easily follows from \cite[Proposition 1.6.7]{BH}.
\end{proof}

\subsection{Parameter sequences}

Let $(R,\fm)$ be a local Noetherian ring and $M$ be a finitely
generated $R$-module. A sequence of elements ${\bf x}$ in $R$ is
said to be a \emph{system of parameters} on $M$ if $M/{\bf x}M$ has
finite length and $\dim M=\ell({\bf x})$. In fact, ${\bf x}$ is a
system of parameters on $M$ if and only if $\h_{M}({\bf
x}R)=\ell({\bf x})=\dim M$.

Using homological properties of the rings instead of height
conditions, the authors in \cite{HM} extended the notion
of system of parameters in Noetherian local
rings to sequences in non-Noetherian ones called strong parameter
sequences. This subsection is devoted to generalize the notion of strong parameter
sequences to modules.

\begin{defn}
A finite sequence ${\bf x}$ of
elements of $R$ is called a parameter sequence on $M$
provided that the following
conditions hold:
\begin{enumerate}
   \item ${\bf x}$ is $M$-weakly proregular,
   \item ${\bf x}M\neq M$,
   \item $H_{{\bf x}}^{\ell({\bf x})}(M)_{\fp}\neq 0$
   for all prime ideals $\fp\in\Supp_{R}(\frac{M}{{\bf x}M})$.
\end{enumerate}
\end{defn}
The sequence ${\bf x}$ is called a \emph{strong parameter
sequence on $M$} if $x_{1},\ldots,x_{i}$
is a parameter sequence on $M$ for $i=1,\ldots,\ell({\bf x})$.
One may consider the empty sequence is a parameter sequence of length zero on
any $R$-module. The empty sequence will also be considered as a
regular sequence of length zero on any $R$-module.

Below, we state some elementary properties of parameter
sequences that will be used in the course of the paper.

\begin{prop}\label{5.3}
Let $R$ be a ring and $M$ be an $R$-module. Let ${\bf x}$ be a finite sequence of elements of $R$.
\begin{enumerate}
  \item Any permutation of a parameter sequence on $M$ is again a parameter sequence on $M$.
  \item Assume that $\sqrt{{\bf x}R}=\sqrt{{\bf y}R}$, $\ell({\bf x})=\ell({\bf y})$, ${\bf x}M\neq
M$ ${\bf y}M\neq M$ and $\Supp_{R}(\frac{M}{{\bf
x}M})=\Supp_{R}(\frac{M}{{\bf y}M})$. Then ${\bf x}$ is a parameter
sequence on $M$ if and only if ${\bf y}$ is a parameter sequence on
$M$.
  \item If $\pgrade_R({\bf x}R,M)=\ell({\bf x})$, then ${\bf x}$ is a parameter sequence on $M$.
  \item Every $M$-regular sequence is a strong parameter sequence on $M$.
  \item Let $f:R\longrightarrow S$ be a flat ring homomorphism. If ${\bf x}$
is a (strong) parameter sequence on $M$ and $\frac{M\otimes_RS}{f({\bf x})M\otimes_RS}\neq 0$,
then $f({\bf x})$ is a (strong) parameter sequence on $M\otimes_RS$. The converse
holds if $f$ is faithfully flat.
\end{enumerate}
\end{prop}
\begin{proof}
For $(1)$ see Proposition \ref{ceck} and Remark \ref{3.2} and for
$(2)$ see, again, Proposition \ref{ceck} together with Theorem
\ref{3.3}.

To prove $(3)$, we first note that $\pgrade_R({\bf
x}^{n}R,M)=\pgrade_R({\bf x}R,M)=\ell({\bf x})$ by \cite[Section
5.5, Theorem 12]{No} and that $H_{i}({\bf x}^{n},M)=0$ for all
$i\geq 1$ by \cite[Proposition 2.7]{HM}. Let $\fp \in
\Supp_{R}(M/{\bf x}M)$. Hence ${\bf x}M_{\fp}\neq M_{\fp}$. Then
$\pgrade_R({\bf x}R_{\fp},M_{\fp})<\infty$ again by
\cite[Proposition 2.7]{HM}. Since localization does not decrease the
polynomial grade by \cite[Section 5.5, Exercise 10]{No} and the
polynomial grade is bounded above by the length of the sequence (see
\cite[Proposition 2.7]{HM}), we see that $\pgrade_R({\bf
x}R_{\fp},M_{\fp})=\ell({\bf x})$ for all $\fp\in \Supp_{R}(M/{\bf
x}M)$. Hence $ H_{{\bf x}}^{\ell}(M)_{\fp}\neq 0$ for all $\fp\in
\Supp_{R}(M/{\bf x}M)$ by \cite[Proposition 2.7]{HM}. Therefore
${\bf x}$ is a parameter sequence on $M$.

For $(4)$, notice that $\ell({\bf x})\geq \pgrade_R({\bf x}R,M)
\geq \ell({\bf x})$ since ${\bf x}M\neq M$. Thus $\pgrade_R({\bf x}R,M)=\ell({\bf x})$.
Therefore ${\bf x}$ is a strong parameter sequence on $M$ by $(3)$.

Finally for (5), assume that ${\bf x}$ is a parameter sequence
on $M$ and $\frac{M\otimes_RS}{f({\bf x})M\otimes_RS}\neq 0$.
Then, by Lemma \ref{fwp}, $f({\bf x})$ is
$M\otimes_RS$-weakly proregular sequence. Now, let
$\fq\in\Supp_S(\frac{M\otimes_RS}{f({\bf x})M\otimes_RS})$ and set $\fp:=f^{-1}(\fq)$. Then
$\fp\in\Supp_{R}(\frac{M}{{\bf x}M})$ and one has the isomorphism
\begin{equation}\label{iso}
H^{\ell({\bf x})}_{f({\bf x})}(M\otimes_RS)_{\fq}\cong H^{\ell({\bf x})}_{{\bf x}}(M)_{\fp}\otimes_{R_{\fp}}S_{\fq}.
\end{equation}
Since $S_{\fq}$ is a faithfully flat $R_{\fp}$-module,
one obtains that $f({\bf x})$ is a parameter sequence on $M\otimes_RS$.
To prove the converse, assume that $f$ is faithfully flat.
Again using Lemma \ref{fwp}, it is enough for us to show
that $H^{\ell({\bf x})}_{{\bf x}}(M)_{\fp}\neq0$ for
all $\fp\in\Supp_{R}(\frac{M}{{\bf x}M})$. Assume
that $\fp\in\Supp_{R}(\frac{M}{{\bf x}M})$. Since $f$
is flat, there exists a prime ideal $\fq$
of $S$ such that $\fp=f^{-1}(\fq)$. The isomorphism
$$(\frac{M\otimes_RS}{f({\bf x})M\otimes_RS})_{\fq}\cong(\frac{M}{{\bf x}M})_{\fp}\otimes_{R_{\fp}}S_{\fq}$$
shows that $\fq\in\Supp_S(\frac{M\otimes_RS}{f({\bf
x})M\otimes_RS})$. The isomorphism (\ref{iso}) now completes the
proof.
\end{proof}

Next we provide a description of parameter sequences using height condition.

\begin{prop}\label{5.4}
Let $R$ be a ring, $M$ be a
finitely generated $R$-module and ${\bf x}$ be a finite sequence of elements of $R$.
\begin{enumerate}
  \item If ${\bf x}$ is a parameter sequence on $M$, then $\h_{M}({\bf x}R)\geq \ell({\bf x})$.
  \item Further assume that $R$ is Noetherian.
Then ${\bf x}$ is a parameter sequence on $M$ if and only if
$\h_{M}({\bf x}R)=\ell({\bf x})$.
\end{enumerate}

\end{prop}
\begin{proof}
(1) If $\h_{M}({\bf x}R)=\infty$, there is nothing to prove. So assume that
$\h_{M}({\bf x}R)< \infty$. Then there exists $\fp\in\Supp(M)\cap
V({\bf x}R)$ such that $\h_{M}({\bf x}R)= \h_{M} \fp=\dim M_{\fp}$. Since ${\bf x}$ is
a parameter sequence on $M$, then $H_{{\bf x}}^{\ell({\bf x})}(M)_{\fp}\neq
0$. Therefore $\ell({\bf x})\leq\dim M_{\fp}=\h_{M}({\bf x}R)$ by Corollary \ref{dim}.

(2) Assume that ${\bf x}$ is a parameter sequence on $M$. Since
${\bf x}M\neq M$, then $\h_{M}({\bf x}R)<\infty$. By part (1), we have
$\h_{M}({\bf x}R)\geq \ell({\bf x})$ and by
Krull's Generalized Principal Ideal Theorem, we have
$\h_{M}({\bf x}R)\leq \h_{R}({\bf x}R) \leq \ell({\bf x})$. Then
$\h_{M}({\bf x}R)=\ell({\bf x})$.

Conversely assume that $\h_{M}({\bf x}R)=\ell({\bf x})$. By Theorem \ref{3.6}, any
sequence of elements in $R$ is $M$-weakly proregular.
Since $\h_{M}({\bf x}R)=\ell({\bf x})<\infty$, then $({\bf x})M\neq M$. Let
$\fp$ be a minimal element of $\Supp_{R}(M/{\bf x}M)$.
Then $\fp$ is a minimal
prime ideal over ${\bf x}R+\Ann(M)$. Hence
$\fp/\Ann(M)$ is a minimal prime ideal over ${\bf x}R(R/\Ann(M))$ which is
generated by $\ell({\bf x})$ elements. Then
$$
\dim M_{\fp}=\h_{M}\fp=\h\frac{\fp}{\Ann(M)}=\h{\bf x}R\frac{R}{\Ann(M)}=\ell({\bf x}).
$$
On the other hand, one has
$$\sqrt{{\bf x}R_{\fp}+\Ann(M)R_{\fp}}=\sqrt{({\bf x}R+\Ann(M))R_{\fp}}=\fp R_{\fp}.$$
Hence
$$
H^{\ell({\bf x})}_{{\bf x}R_{\fp}}(M_{\fp})=H^{\ell({\bf x})}_{{\bf
x}R_{\fp}+ \Ann(M)R_{\fp}}(M_{\fp})=H^{\ell({\bf x})}_{\fp
R_{\fp}}(M_{\fp})\neq0.
$$
Therefore ${\bf x}$ is a parameter sequence on $M$.
\end{proof}

The Noetherian assumption in Proposition \ref{5.4}(2) is crucial.
In fact, in every valuation domain of dimension 2, one can
choose a weakly proregular sequence $x, y$ such that $\h(x,y)=2$,
but $x,y$ is not a parameter sequence, see \cite[Example 3.7]{HM}.

\section{Cohen-Macaulay Modules}
\subsection{Definition and basic properties}
In \cite{G0} and \cite{G}, Glaz raised the question that
whether there exists a generalization of the notion of
Cohen-Macaulayness with certain desirable properties to
non-Noetherian rings. One of those is that every coherent
regular ring is Cohen-Macaulay. In this direction,
in \cite{HM}, it is defined a notion of
Cohen-Macaulayness for arbitrary commutative rings.
This subsection is devoted to extend the
definition of Cohen-Macaulayness
for commutative rings in the sense of \cite{HM} to modules.
\begin{defn}\label{CM}
An $R$-module $M$ is called a Cohen-Macaulay $R$-module if every strong parameter
sequence on $M$ is an $M$-regular sequence.
\end{defn}
This definition agrees with the usual definition of
Cohen-Macaulay finitely generated modules over Noetherian rings. Indeed, let
$R$ be a Noetherian ring and $M$ be a finitely
generated $R$-module. Assume that $M$ is Cohen-Macaulay
in the sense of Definition \ref{CM}.
To show that $M$ is Cohen-Macaulay with the usual definition
in the Noetherian case, it is enough to show that $\grade(I,M)=\h_M I$ for all proper
ideals $I$ of $R$. To prove this, assume that $I$ is a proper
ideal of $R$ and set $\h_M I=\ell$. Since $\h_M I=\h I(R/\Ann(M))$,
employing \cite[Theorem A.2]{BH} to the ring $R/\Ann(M)$ one
finds the elements $x_1,\ldots,x_{\ell}$ in $I$
such that $\h_M (x_1,\ldots,x_i)=i$ for all $i=0,\ldots,\ell$.
It follows from Proposition \ref{5.4}
that $x_1,\ldots,x_{\ell}$ is a strong parameter
sequence on $M$. Hence it is an $M$-regular sequence. This yields
that $\ell\leq\grade(I,M)\leq\h_M I=\ell$. Therefore $\grade(I,M)=\h_M I$.
The converse is true by Theorem \ref{6.2} and
Proposition \ref{6.5} below and \cite[Corollary 1.6.19]{BH}.

Let $R$ be a ring and $M$ be an  $R$-module. If $\dim M=0$, then $M$ is Cohen-Macaulay.
Indeed, in this situation, $M$ has not any parameter sequences.

Thanks to polynomial grade, Koszul homology, and \v{C}ech cohomology
of strong parameter sequences, our first result presents some
equivalent statements of Cohen-Macaulayness. It generalizes
\cite[Proposition 4.2]{HM} for modules. Its proof is mutatis
mutandis the same as that of \cite[Proposition 4.2]{HM}. But,
for the reader's convenience, we reprove it in the case of modules.

\begin{thm}\label{6.2}
Let $R$ be a ring and $M$ be an $R$-module. The following conditions are
equivalent:
\begin{enumerate}
  \item $M$ is Cohen-Macaulay.
  \item $\grade({\bf x}R,M)=\ell({\bf x})$ for every strong parameter sequence ${\bf x}$ of $M$.
  \item $\pgrade_R({\bf x}R,M)=\ell({\bf x})$ for every strong parameter sequence ${\bf x}$ of $M$.
  \item $H_{i}({\bf x},M)=0 $ for all $i\geq 1$ and for every strong parameter sequence ${\bf x}$ of $M$.
  \item $H_{{\bf x}}^{i}(M)=0 $ for all $i<\ell({\bf x})$ and for every strong parameter sequence ${\bf x}$ of $M$.
\end{enumerate}
\end{thm}
\begin{proof}
$(1)\Rightarrow(2)$ Assume that ${\bf x}$ is a strong
parameter sequence on $M$; so, by assumption, ${\bf x}$ is $M$-regular sequence.
Hence $\ell({\bf x})\leq \grade({\bf x}R,M)$. One also
notices that $\grade({\bf x}R,M)\leq \pgrade_R({\bf x}R,M)$ by \cite[Page 149]{No} and that
$\pgrade_R({\bf x}R,M)\leq\ell({\bf x})$
by \cite[Section 5.5, Theorem 13]{No}. Therefore one has $\grade({\bf x}R,M)=\ell({\bf x})$.

$(2)\Rightarrow(3)$  Assume that ${\bf x}$
is a strong parameter sequence on $M$; so that
$\ell({\bf x})=\grade({\bf x}R,M)$. As in the $(1)\Rightarrow(2)$,
again using \cite[Page 149]{No} and \cite[Section 5.5, Theorem 13]{No},
one obtains that $\pgrade_R({\bf x}R,M)=\ell({\bf x})$.

$(3)\Rightarrow(1)$  Assume that ${\bf x}$ is a strong
parameter sequence on $M$. We proceed by induction on $\ell=\ell({\bf x})$ to show
that ${\bf x}$ is $M$-regular sequence. If
${\bf x}=x_1$, then $\pgrade_R(x_1R,M)=1$ and $0=H_{1}(x_1,M)=(0:_{M}x_1)$.
Hence $x_1$ is an $M$-regular element. Suppose that every strong
parameter sequence on $M$ of length at most $\ell-1$ is $M$-regular
sequence and that ${\bf x}$ is a strong parameter sequence on $M$ of
length $\ell$. Set ${\bf x'}=x_{1},\ldots,x_{\ell-1}$. Since
${\bf x'}$ is a strong parameter sequence on $M$, then by hypothesis
$\pgrade_R({\bf x'},M)=\ell-1$. Thus, by induction, ${\bf x'}$ is
an $M$-regular sequence. Let $M'=M/{\bf x'}M$. Since $\pgrade_R({\bf x},M)=\ell$, then
$H_{\ell-i}({\bf x},M)=0$ for all $i<\ell$ by \cite[Proposition 2.7]{HM}. Hence
$$
(0:_{M'}x_{\ell})=H_{1}(x_{\ell},M')=H_{1}(x_{\ell},M/{\bf x'}M)\cong
H_{1}({\bf x},M)=0
$$ by
\cite[Proposition 1.6.13]{BH}. This implies that $x_{\ell}$ is an
$M'$-regular element. Therefore ${\bf x}$ is an $M$-regular sequence.
Finally, notice that $(3)$, $(4)$ and
$(5)$ are equivalent by \cite[Proposition 2.7]{HM}.
\end{proof}

The Cohen-Macaulay property descends along faithfully flat extensions:

\begin{prop}\label{flat}
Let $f:R\to S$ be a faithfully flat ring homomorphism.
Let $M$ be an $R$-module. If $M\otimes_{R}S$ is Cohen-Macaulay
$S$-module, then $M$ is Cohen-Macaulay.
\end{prop}
\begin{proof}
Assume that $M\otimes_{R}S$ is Cohen-Macaulay $S$-module. Let ${\bf x}$
be a strong parameter sequence on $M$. Then, by Proposition
\ref{5.3}(5), $f({\bf x})$ is strong parameter
sequence on $M\otimes_{R}S$. Hence, the assumption
together with Theorem \ref{6.2}, \cite[Section 5.5, Theorem 19]{No}
and \cite[Proposition 2.7]{HM} yields that \begin{align*}
 \ell({\bf x})
 & = \ell(f({\bf x}))\\
 & = \pgrade_R(f({\bf x})S,M\otimes_{R}S) \\
 & = \pgrade_R({\bf x}R,M\otimes_{R}S) \\
 & =\sup\{k\geq0\mid H^{i}_{{\bf x}}(M\otimes_{R}S)=0\text{ for  all }i<k \} \\
 & =\sup\{k\geq0\mid H^{i}_{{\bf x}}(M)\otimes_{R}S=0\text{ for  all }i<k \} \\
 & =\sup\{k\geq0\mid H^{i}_{{\bf x}}(M)=0\text{ for  all }i<k \} \\
 & = \pgrade_R({\bf x}R,M).
\end{align*}
Therefore $M$ is Cohen-Macaulay.
\end{proof}

One can immediately obtain the following corollaries.

\begin{cor}
Let $(R,\fm)$ be a quasi-local ring and $M$ an $R$-module. If
$M\otimes_{R}\widehat{R}$ is Cohen-Macaulay $\widehat{R}$-module
where $\widehat{R}$ is the $\fm$-adic completion of $R$,
then $M$ is Cohen-Macaulay.
\end{cor}
\begin{cor}
Let $R$ be a ring, $M$ an $R$-module and $t$ an indeterminate over
$R$. If $M\otimes_RR[t]$ is Cohen-Macaulay $R[t]$-module, then $M$
is Cohen-Macaulay.
\end{cor}

\begin{prop}\label{6.5}
Let $R$ be a ring and $M$ be a finitely generated $R$-module. If $M _{\fm}$ is
Cohen-Macaulay $R_{\fm}$-module for all maximal ideals $\fm$ of $R$,
then $M$ is a Cohen-Macaulay $R$-module.
\end{prop}
\begin{proof}
Assume that ${\bf x}=x_{1},\ldots,x_{\ell}$ is a strong parameter sequence on
$M$ and that $\fm$ is a maximal ideal containing ${\bf x}R$.
Since $M$ is finitely generated, then $M_{\fm}\neq (\frac{\bf x}{1})M_{\fm}$; so that
$\frac{\bf x}{1}=\frac{x_{1}}{1},\ldots,\frac{x_{\ell}}{1}$ is a strong parameter sequence
on $M _{\fm}$. Hence $\frac{\bf x}{1}$ is $M_{\fm}$-regular
sequence. Thus $\bf x$ is $M$-regular sequence. Therefore $M$ is Cohen-Macaulay.
\end{proof}

As mentioned in the introduction, every coherent regular ring is
Cohen-Macaulay. In particular, every valuation domain is
Cohen-Macaulay. In the following, we show that every torsion-free
module over such domain is Cohen-Macaulay. In fact, we do this for
torsion-free modules over almost valuation domains. Recall that an
integral domain $R$ with quotient field $K$ is called an {\em almost
valuation domain} if for every nonzero $x\in K$, there exists an
integer $n\geq 1$ such that either $x^n\in R$ or $x^{-n}\in R$
\cite{AZ}.

\begin{prop}
Every torsion-free module over an almost valuation domain is
Cohen-Macaulay.
\end{prop}
\begin{proof}
Suppose that $(R,\fm)$ is an almost valuation domain and $M$ is a
torsionfree $R$-module. Assume that ${\bf x}:=x_{1},x_{2}$ is a
sequence in $R$ of length 2. Assume that $x_1^nR\subseteq x_2^nR$
for some positive integer $n$. Then
$$H^{2}_{\bf x}(M)=H^{2}_{x_{1}^{n},x_{2}^{n}}(M)=H^{2}_{x_{2}^{n}}(M)=0.$$
Hence ${\bf x}=x_{1},x_{2}$ can not be a parameter sequence on $M$.
Then, for each parameter sequence ${\bf x}$ of $M$, $\ell(\bf
x)\leq1$. Therefore $M$ is Cohen-Macaulay. To this end, one notices
that $M$ is torsion-free.
\end{proof}

Recall that the module $M$ is
called \emph{Cohen-Macaulay in the
sense of ideals (resp. finitely generated ideals)} if
$\h_M(I)=\pgrade_R(I,M)$ for all ideals
(resp. finitely generated ideals) $I$, see \cite[Definition 3.1]{AT}.
The following proposition generalizes
\cite[Theorem 3.4]{AT} to finitely generated modules.

\begin{prop}
Let $R$ be a ring and $M$ be a finitely generated $R$-module. If $M$ is
Cohen-Macaulay in the sense of ideals (or finitely generated ideals),
then $M$ is Cohen-Macaulay  in the sense of Definition \ref{CM}.
\end{prop}
\begin{proof}
Assume that ${\bf x}$ is a strong parameter sequence on $M$. Then, by
Proposition \ref{5.4} and \cite[Section 5.5, Theorem 13]{No}, we have
$$\h_{M}({\bf x}R)\geq \ell({\bf x}) \geq \pgrade_R({\bf x}R,M).$$
However, by assumption $\pgrade_R({\bf x}R,M)=\h_{M}({\bf x}R)$. Then
$\pgrade_R({\bf x}R,M)=\ell({\bf x})$.
Therefore $M$ is Cohen-Macaulay  in the sense of Definition \ref{CM}.
\end{proof}

\subsection{The Cohen-Macaulayness of some constructions}

Let $R$ and $S$ be two commutative rings with unity, let $J$
be an ideal of $S$ and $f:R\to S$  be a ring
homomorphism. The subring
$R\bowtie^{f}J:=\{(x,f(x)+j) | x\in R \text{ and } j\in J\}$
of $R\times S$ is called the \emph{amalgamation of R with S along J
with respect to f} \cite{DFF}. This construction
generalizes several classical constructions. Among them is the Nagata's
trivial extension, see \cite[Examples 2.5 and 2.6]{DFF}.

Under mild conditions, the next proposition shows that
the Cohen-Macaulayness of $R\bowtie^{f}J$ descends to that of $R$.

\begin{prop}
Let $R$ and $S$ be commutative rings with unity, let $J$
be an ideal of $S$ and $f:R\to S$  be a ring
homomorphism. Assume that $J$ is flat as an
$R$-module induced by $f$. If $R\bowtie^{f}J$ is
Cohen-Macaulay, then $R$ is Cohen-Macaulay. Moreover,
if any strong parameter sequence on $J$ is a strong
parameter sequence on $R$, then $J$ is Cohen-Macaulay.
\end{prop}
\begin{proof}
Note that as an $R$-module $R\bowtie^{f}J\cong R\oplus J$.
This in conjunction with the assumption implies the natural
embedding $\iota_R:R\longrightarrow R\bowtie^{f}J$ is faithfully
flat. Hence, by Proposition \ref{flat}, $R$ is Cohen-Macaulay.
The rest of the conclusion is clear since $J$ is flat.
\end{proof}

We do not know whether the Cohen-Macaulay property of $R$ ascends to
that of $R\bowtie^{f}J$. The difficulty lies in linking strong
parameter sequences of $R\bowtie^{f}J$ to strong parameter sequences
of $R$ and $J$. However, we solve this difficulty under certain
assumptions. Let $M$ be an $R$-module. Then $R\ltimes M$ denotes the
\emph{trivial extension} of $R$ by $M$. As indicated in
\cite[Example 2.8]{DFF}, if $S:=R\ltimes M$, $J:=0\ltimes M$, and
$f:R\to S$ be the natural embedding, then $R\bowtie^f J\cong
R\ltimes M$.

Assume that $R$ is Noetherian local and that $M$
is finitely generated. It is well known that the trivial extension
$R\ltimes M$ is Cohen-Macaulay if and only if $R$ is Cohen-Macaulay
and $M$ is maximal Cohen-Macaulay, see \cite[Corollary 4.14]{AW}.
Our main theorem in this paper generalizes this result.
To prove it we need the following lemma.

\begin{lem}\label{w}
Let $R$ be a ring and $M$ be an $R$-module. Set $S:=R\ltimes M$.
Let $\pi:S\to R$ be the natural projection. Then ${\bf x}\subseteq S$ is
$S$-weakly proregular if and only if $\pi({\bf x})$ is $R$ and $M$-weakly proregular.
\end{lem}
\begin{proof}
First of all notice that using the structure of prime spectrum
of the trivial extension one has $\sqrt{{\bf x}S}=\sqrt{\pi({\bf x})S}$.
Using this together with Proposition \ref{ceck}(2) and Theorem \ref{3.3},
one can deduce that ${\bf x}$ is $S$-weakly proregular if and only
if $(\pi({\bf x}),0)$ is $S$-weakly proregular. Thus, it suffices to prove
that if ${\bf y}$ is a finite sequence of elements from $R$,
then ${\bf y}$ is $R$ and $M$-weakly proregular if and only
if $({\bf y},0)$ is $S$-weakly proregular. However, as $R$-modules,
$H_i(({\bf y}^n,0),S)\cong H_i({\bf y}^n,R)\oplus H_i({\bf y}^n,M)$ for all $i$ and $n$.
Therefore the definition of weakly proregular sequence completes the proof.
\end{proof}

\begin{thm}\label{th}
Let $R$ be a ring and $M$ be an $R$-module such that every
$R$-weakly proregular sequence is an $M$-weakly proregular sequence.
Then $R\ltimes M$ is Cohen-Macaulay if and only if $R$ is
Cohen-Macaulay and every $R$-regular sequence is a weak $M$-regular
sequence.
\end{thm}
\begin{proof}
First assume that $R\ltimes M$ is Cohen-Macaulay and that ${\bf x}$
is a strong parameter sequence on $R$. So, in particular, ${\bf x}$
is $M$-weakly proregular. By Lemma \ref{w}, $({\bf x},0)\subseteq
R\ltimes M$ is $R\ltimes M$-weakly proregular. Also, one notices
that
$$\frac{R\ltimes M}{({\bf x},0)R\ltimes M}=\frac{R\ltimes M}{{\bf x}(R\ltimes M)}\cong\frac{R}{{\bf x}R}\ltimes\frac{M}{{\bf x}M}\neq0.$$
Finally, for a prime ideal $\fp\ltimes M$ of $R\ltimes M$ containing
$({\bf x},0)$, using Proposition \ref{ceck} and \cite[Exercise
6.2.12(iv)]{BSh}, one has
$$H^{\ell({\bf x})}_{({\bf x},0)}(R\ltimes M)_{\fp\ltimes M}\cong H^{\ell({\bf x})}_{{\bf x}}(R)_{\fp}\oplus H^{\ell({\bf x})}_{{\bf x}}(M)_{\fp}\neq0.$$

Therefore $({\bf x},0)$ is a strong parameter sequence on $R\ltimes
M$. So that $({\bf x},0)$ is a $R\ltimes M$-regular sequence. It
easily follows that ${\bf x}$ is $R$-regular sequence and that ${\bf
x}$ is weak $M$-regular sequence. This shows that $R$ is
Cohen-Macaulay and every $R$-regular sequence is a weak $M$-regular
sequence.

Conversely assume that $R$ is Cohen-Macaulay and every $R$-regular
sequence is a weak $M$-regular sequence. Let ${\bf x}$ be a strong
parameter sequence on $R\ltimes M$. Then, using Lemma \ref{w},
$\pi({\bf x})\subseteq R$ is $R$ and $M$-weakly proregular sequence.
Also notice that
$$R\ltimes M\neq\sqrt{{\bf x}(R\ltimes M)}=\sqrt{\pi({\bf x})(R\ltimes M)}=\sqrt{\pi({\bf x})R}\ltimes M.$$
Hence $R\neq\pi({\bf x})R$. Now let $\fp$ be a prime ideal of $R$
containing $\pi({\bf x})$. So that the prime ideal $\fp\ltimes M$ of
$R\ltimes M$ contains ${\bf x}$. Hence $H^{\ell({\bf x})}_{\bf
x}(R\ltimes M)_{\fp\ltimes M}\neq 0$. On the other hand, again using
Proposition \ref{ceck} and \cite[Exercise 6.2.12(iv)]{BSh}, one has
$$H^{\ell({\bf x})}_{\bf x}(R\ltimes M)_{\fp\ltimes M}\cong H^{\ell({\bf x})}_{\pi({\bf x})}(R\ltimes M)_{\fp\ltimes M}
\cong H^{\ell({\bf x})}_{\pi({\bf x})}(R)_{\fp} \oplus H^{\ell({\bf
x})}_{\pi({\bf x})}(M)_{\fp}.$$ Consequently, Proposition
\ref{ceck}(4) yields that $H^{\ell({\bf x})}_{\pi({\bf
x})}(R)_{\fp}\neq0$. This means that $\pi({\bf x})$ is a strong
parameter sequence on $R$; hence $\pi({\bf x})$ is $R$-regular
sequence and so by our assumption $\pi({\bf x})$ is a weak
$M$-regular sequence. Then $H^{i}_{\pi({\bf x})}(R)=0$ for all
$i<\ell({\bf x})$ by Theorem \ref{6.2}. Note that $\pi({\bf x})^n$
is a weak $M$-regular sequence for all positive integer $n$
(\cite[Exercise 1.1.10]{BH}), we have $H_{\ell({\bf x})-i}({\pi({\bf
x})}^n,M)=0$ for all $i<\ell({\bf x})$ (\cite[Theorem 1.6.16]{BH}).
Hence in view of \cite[Proposition 1.6.10(d)]{BH} $H^{i}_{\pi({\bf
x})}(M)=\varinjlim H^{i}({\pi({\bf x})}^n,M)=0$ for all $i<\ell({\bf
x})$. Therefore
$$H^{i}_{\bf x}(R\ltimes M)\cong H^{i}_{\pi({\bf x})}(R\ltimes M)\cong H^{i}_{\pi({\bf x})}(R)\oplus H^{i}_{\pi({\bf
x})}(M)=0$$ for all $i<\ell({\bf x})$. It now follows from Theorem
\ref{6.2} that $R\ltimes M$ is Cohen-Macaulay.
\end{proof}

We note that the backward direction in Theorem \ref{th} holds in
general, without supposing that every $R$-weakly proregular sequence
is an $M$-weakly proregular sequence.

Let $R$ be a ring and $M$ be an $R$-module. It can be seen from the
proof of the above theorem that $R\ltimes M$ is Cohen-Macaulay if
and only if every strong parameter sequence on $R\ltimes M$ of the
form $({\bf x},0)$ is a regular sequence.

\begin{cor}
(c.f. \cite[Corollary 4.14]{AW}) Let $R$ be a Noetherian local ring
and $M$ be a finitely generated nonzero $R$-module. The following
statements are equivalent:
\begin{enumerate}
  \item $R\ltimes M$ is Cohen-Macaulay;
  \item $R$ is Cohen-Macaulay and every $R$-regular sequence is an
  $M$-regular sequence;
  \item $R$ is Cohen-Macaulay and $M$ is maximal Cohen-Macaulay.
\end{enumerate}
\end{cor}

\begin{center} {\bf ACKNOWLEDGMENT}

\end{center}
The authors would like to thank deeply the referee for her/his
careful reading and for some substantial comments.

\end{document}